\documentclass[12pt]{article}
\usepackage{amsmath,amssymb,amsthm, amsfonts}
\usepackage{etaremune, enumerate, float, verbatim}
\usepackage{hyperref}
\usepackage{graphicx}
\usepackage{url}
\usepackage[mathlines]{lineno}
\usepackage{dsfont} 
\usepackage{tikz, graphicx, subcaption, caption}
\usepackage[algo2e,ruled,vlined]{algorithm2e}
\usepackage{mathrsfs,amsmath}
\usepackage{color}
\definecolor{red}{rgb}{1,0,0}

\definecolor{blue}{rgb}{0,0,.7}

\definecolor{green}{rgb}{0,.6,0}

\definecolor{purp}{rgb}{.5,0,.5}

\numberwithin{figure}{section}   

\newtheorem{thm}{Theorem}[section]

\newtheorem{lem}[thm]{Lemma}

\newtheorem{quest}[thm]{Question}

\theoremstyle{definition}

\theoremstyle{definition}

\theoremstyle{definition}

\newcommand{\rad}{\operatorname{rad}}

\newcommand{\bit}{\begin{itemize}}
\newcommand{\eit}{\end{itemize}}
\newcommand{\ben}{\begin{enumerate}}
\newcommand{\een}{\end{enumerate}}
\newcommand{\beq}{\begin{equation}}
\newcommand{\eeq}{\end{equation}}
\newcommand{\bea}{\begin{eqnarray*}} 
\newcommand{\eea}{\end{eqnarray*}}
\newcommand{\bpf}{\begin{proof}}
\newcommand{\epf}{\end{proof}\ms}
\newcommand{\bmt}{\begin{bmatrix}}
\newcommand{\emt}{\end{bmatrix}}
\newcommand{\ms}{\medskip}

\title{Applications of the abc conjecture to powerful numbers}
\author{P.A. CrowdMath}

\begin{document}
\maketitle


\begin{abstract}
The abc conjecture is one of the most famous unsolved problems in number theory. The conjecture claims for each real $\epsilon > 0$ that there are only a finite number of coprime positive integer solutions to the equation $a+b = c$ with $c > (\rad(a b c))^{1+\epsilon}$. If true, the abc conjecture would imply many other famous theorems and conjectures as corollaries. In this paper, we discuss the abc conjecture and find new applications to powerful numbers, which are integers $n$ for which $p^2 | n$ for every prime $p$ such that $p | n$. We answer several questions from an earlier paper on this topic, assuming the truth of the abc conjecture.
\end{abstract}

\section{Introduction}\label{s:intro}

For each positive integer $n$ define the \emph{radical} $\rad(n)$ of $n$ to be the product of the distinct prime divisors of $n$. In 1985, David Masser conjectured \cite{mas} that for every positive real number $\epsilon$, there are only finitely many triples $(a, b, c)$ of coprime positive integers with $a+b = c$ such that $c > (\rad(a b c))^{1+\epsilon}$. This conjecture has come to be called the \emph{abc conjecture}, and a proof of the abc conjecture would affirm many other famous conjectures and theorems in number theory. For example, both Fermat's Last Theorem for sufficiently large powers and Roth's Theorem on diophantine approximation of algebraic numbers are corollaries of the abc conjecture \cite{van}, as are Lang's conjecture on the N\'{e}ron-Tate height and the Erd\H{o}s-Woods conjecture on sequences of consecutive integers.

Call a positive integer $x$ a \emph{powerful number} if for every prime $p$ such that $p | x$, we also have $p^2 | x$. Call $x$ a \emph{$k$-powerful number} if $p^k | x$ for every prime $p$ with $p | x$. In \cite {cupa}, Cushing and Pascoe found new ways to apply the abc conjecture to prove conditional results about powerful numbers. In addition to proving new corollaries of the abc conjecture, they also asked several questions about powerful numbers. In this paper we answer four of the questions from \cite{cupa}, assuming the truth of the abc conjecture. In particular, we prove that there are only finitely many solutions to the equation $x+y= z$ with $(x, y) = 1$ and $x, y, z$ all $4$-powerful, solving Problem 2 from \cite{cupa}. We also show for $n \geq 5$ that there are only finitely many powerful numbers of the form $x^n+y^n$ where $(x, y) = 1$, partially solving Problem 3 from \cite{cupa} for $n \geq 5$. For $n \leq 3$, there are infinitely many powerful numbers of the form $x^n+y^n$ with $(x, y) = 1$. Problem 3 is not completely solved, since the $n = 4$ case is still open.

We solve Problem 4 from \cite{cupa} by proving that there are only finitely many powerful numbers of the form $k^n+r$ for fixed positive integers $k$ and $r$ with $(k, r) = 1$. We also solve Problem 5 from \cite{cupa} by showing that for any fixed positive integers $r$ and $k$, the numbers $(n!)^r + k$ are powerful only finitely often. In addition, we present the solution for another problem from the author of \cite{cupa} about powerful numbers that was posted on the CrowdMath 2018 forum. In the last section, we discuss several open problems. 

\section{Results on powerful numbers}

For the remainder of the paper, we assume that the abc conjecture is true. We start with a solution to Problem 2 in \cite{cupa}.

\begin{thm}
There are only finitely many solutions to the equation $x+y= z$ with $(x, y) = 1$ and $x, y, z$ all $4$-powerful.
\end{thm}

\begin{proof}
Let $a = x$, $b = y$, $c = z$, and $\epsilon = \frac{1}{3}$. Observe that $$(\rad (xyz))^{\frac{4}{3}} \le (\rad(x) \rad(y) \rad(z) )^{\frac{4}{3}}\le (x^{\frac{1}{4}}y^{\frac{1}{4}}z^{\frac{1}{4}})^{\frac{4}{3}} = x^{\frac{1}{3}}y^{\frac{1}{3}}z^{\frac{1}{3}}$$ which is clearly less than $z$. Thus by the abc conjecture, there are only finitely many solutions to the equation $x+y= z$ with $(x, y) = 1$ and $x, y, z$ all $4$-powerful.
\end{proof}

The next theorem partially solves Problem 3 from \cite{cupa}, which was to determine when $x^n+y^n$ is powerful, with $(x,y) = 1$. Note that for $n = 2$ and $n = 3$, there are infinitely many powerful numbers of the form $x^n+y^n$ with $x,y$ coprime, since there are infinitely many coprime solutions to $x^2+y^2 = z^2$ and $x^3+y^3 = z^2$. The problem is still open for $n = 4$.

\begin{thm}
For $n \geq 5$, there are only finitely many powerful numbers of the form $x^n+y^n$ where $(x, y) = 1$.
\end{thm}

\begin{proof}
Let $x^n + y^n = z$ where $z$ is a powerful number and $(x,y) = 1$. Without loss of generality, we let $x \le y$. Then, we have that$$(xy)^{10} = (\sqrt{x^5y^5})^4 \le \frac{(x^5 + y^5)^4}{16} <  (x^5 + y^5)^4 < (x^n + y^n )^4= z^4.$$

Now, we apply the $abc$ conjecture. We let $a = x^n$, $b = y^n$, and $c = z$ for relatively prime $a$, $b$, and $c$. Letting $\epsilon = \frac{1}{9}$, we see that$$(\rad(x^ny^nz))^{\frac{10}{9}} \le (\rad (x^n))^{\frac{10}{9}}( \rad (y^n))^{\frac{10}{9}} (\rad (z))^{\frac{10}{9}} \le $$$$(\rad (x))^{\frac{10}{9}} (\rad (y))^{\frac{10}{9}} (z^{\frac{1}{2}})^{\frac{10}{9}} \le (xy)^{\frac{10}{9}}z^{\frac{5}{9}} < z.$$ Thus for $n \geq 5$, there are a finite number of solutions.
\end{proof}

Next we solve Problem 4 from \cite{cupa}, which was to determine whether $2^n+1$ is powerful only finitely often. We prove the more general result that there are a finite number of powerful numbers of the form $k^n + r$ for any fixed positive integers $k$ and $r$ with $(k, r) = 1$.

\begin{thm}
If $k$ are $r$ are fixed positive integers with $(k, r) = 1$, then there are a finite number of powerful numbers of the form $k^n + r$.
\end{thm}

\begin{proof}
Suppose that $k^n+r = z$ for powerful $z$. Let $a = k^n$, $b = r$, $c = z$, and $\epsilon = \frac{1}{9}$. Observe that$$(\rad (k^nrz))^{\frac{10}{9}} \le (\rad (k^n))^{\frac{10}{9}}(\rad (r))^{\frac{10}{9}}(\rad (z))^{\frac{10}{9}} \le $$$$(\rad (k))^{\frac{10}{9}}(\rad (r))^{\frac{10}{9}}(z^{\frac{1}{2}})^{\frac{10}{9}} \le k^{\frac{10}{9}}r^{\frac{10}{9}}z^{\frac{5}{9}}$$ which is less than $z$ for $n$ sufficiently large, since $k$ and $r$ are fixed. Therefore if the abc conjecture is true, then for any fixed $k \ge 2$ and $r \geq 1$, there must be a finite number of powerful numbers of the form $k^n + r$. For $k = 1$, we have only one possible value of $k^n + r$. Therefore, for any fixed positive $k$ and $r$ with $(k, r) = 1$, there are a finite number of powerful numbers of the form $k^n + r$.
\end{proof}

The next result solves Problem 5 from \cite{cupa}. For the proof of this result, we define the \emph{primorial of $x$} denoted $x\# $ to be the product of all distinct primes that are at most $x$.

\begin{thm} 
For any fixed positive integers $r$ and $k$, the numbers $(n!)^r + k$ are powerful only finitely often.
\end{thm}

\begin{proof}
Let $(n!)^r + k = z$ for powerful $z$, and let $n \geq k$, so that $k | n!$ and the equation becomes $\frac{(n!)^r}{k} + 1 = \frac{z}{k}$. Let $a = \frac{(n!)^r}{k}$, $b = 1$, and $c = \frac{z}{k}$. Observe that: $$\rad (\frac{(n!)^r}{k} )\leq \rad ({(n!)}^r) = \rad (n!) \leq n\#$$
 $$\rad (\frac{z}{k}) \leq \rad (z) \leq z^{\frac{1}{2}}$$ 
$$\rad (abc) = \rad (\frac{(n!)^r}{k} \frac{z}{k}) \leq (n\#)(z^{\frac{1}{2}})$$ So, letting
$\epsilon = \frac{1}{2}$, $(\rad (abc))^{\frac{3}{2}} \leq (n\#)^{\frac{3}{2}} z^{\frac{3}{4}} < z = c$ for $n$ sufficiently large, since $n \# \leq 4^n$ and $z > n!$. Thus, $(n!)^r + k$ is powerful only finitely often, if the abc conjecture is true.
\end{proof}

During CrowdMath 2018, Cushing proposed two more problems about powerful numbers and the abc conjecture on the CrowdMath forum \cite{cpp}. We solved one of the problems and made some progress on the other \cite{cpp, cppp}. The next lemma and theorem together solve the first problem by showing that $lim inf |a_{n+1}-a_{n}| = \infty$, where $a_1<a_2<a_3 \dots$ denotes the set of $3$-powerful numbers. We discuss the other problem in the next section.

\begin{lem}
Let $M, k$ be fixed positive integers. Then there are finitely many $x\in\mathbb{Z}_{>0}$ such that $p<M$ for every prime $p$ dividing $x(x+k)$.
\end{lem}

\begin{proof}
It suffices to prove that there are finitely many $x\in\mathbb{Z}_{>0}$ with $(x,k)=1$ such that $p<M$ for every prime $p$ dividing $x(x+k)$. Let $x$ be an integer with $x>(kM!)^2$. Then we have
$$\rad(kx(x+k))<k\cdot\prod_{p<M}p\le k\cdot M!<\sqrt{x+k}\implies x+k>\rad(kx(x+k))^2$$.

By the abc conjecture with $a=x, b=k, c=x+k$ and $\epsilon=1$, there are finitely many $x\in\mathbb{Z}_{>0}$ such that $p<M$ for every prime $p$ dividing $x(x+k)$.
\end{proof}

\begin{thm}
If $a_1<a_2<a_3 \dots$ denotes the set of $3$-powerful numbers, then $lim inf |a_{n+1}-a_{n}| = \infty$.
\end{thm}

\begin{proof}
Fix $k$ and suppose that $x$ and $y$ are $3$-powerful numbers with $x=y+k$ and $(y,k)=d$. Observe that
$$\rad(\frac{xyk}{d^3})\le \rad(x)\rad(y)\rad(k) \le \rad(k)\sqrt[3]{xy}<x^{\frac{2}{3}}\rad(k)$$But this implies that
$$\frac{x}{d}>\frac{1}{d\cdot\rad(k)^{1.5}}(\rad(\frac{xyk}{d^3}))^{1.5}$$ Combining with the last lemma, we choose $y$ large enough so that
$$\frac{x}{d}>\frac{1}{d\cdot\rad(k)^{1.5}}(\rad(\frac{xyk}{d^3}))^{1.5}>(\rad(\frac{xyk}{d^3}))^{1.4}$$ Hence by the $abc$ conjecture with $a=\frac{y}{d}, b=\frac{k}{d}, c=\frac{x}{d}$ and $\epsilon=0.4$, there are only finitely many $y$ where both $y$ and $y+k$ are $3$-powerful numbers.
\end{proof}

\section{Concluding remarks and open problems}

We note that in most of our results, the set of powerful numbers can be replaced with the set of numbers which are at least their radical squared, and the set of $k$-powerful numbers can be replaced with the set of numbers which are at least their radical to the $k^{th}$ power. The following problem about powerful numbers was proposed during CrowdMath 2018. There was some progress on this problem for some specific polynomials \cite{cppp}, but no progress in general.

\begin{quest}
Let $P$ be a polynomial with integer coefficients and at least $3$ simple roots. Is it true that $P(n)$ is powerful only finitely often?
\end{quest}

Although we answered most of the questions from \cite{cupa}, Problem 1 in that paper is still unsolved. 

\begin{quest} \cite{cupa} Does every coprime arithmetic progression contain infinitely many powerful pairs?
\end{quest}

We also answered Problem 3 from \cite{cupa} for all $n$ except $n = 4$.

\begin{quest}
Is $x^4+y^4$ powerful only finitely often when $(x, y) = 1$?
\end{quest}

\section{Acknowledgments}

CrowdMath is an open program created by the MIT Program for Research in Math, Engineering, and Science (PRIMES) and Art of Problem Solving that gives high school and college students all over the world the opportunity to collaborate on a research project. The 2018 CrowdMath project is online at http://www.artofproblemsolving.com/polymath/mitprimes2018.

\end{document}